\documentclass[11pt]{article}

\usepackage{amsmath,amsthm,amsfonts,amssymb,mathtools}
\usepackage{enumitem}

\usepackage[margin=2.3cm,nohead]{geometry}

\usepackage{url}
\usepackage{hyperref}
\usepackage{graphicx}
\usepackage{grffile}
\usepackage{float}
\usepackage{caption}
\usepackage{subcaption}
\graphicspath{{imgs/}}

\usepackage{booktabs}
\usepackage{multirow}

\usepackage{algorithm}
\usepackage{algpseudocode}

\usepackage[
  backend=biber,
  style=ieee,
  citestyle=numeric-comp,
  url=false,
  isbn=false,
  sortcites=true
]{biblatex}
\setlist[enumerate]{label=(\roman*), align=left}

\newtheorem{theorem}{Theorem}
\newtheorem{lemma}[theorem]{Lemma}

\newtheorem{proposition}[theorem]{Proposition}
\newtheorem{definition}{Definition}
\newtheorem{assumption}{Assumption}
\newtheorem{remark}{Remark}

\newcommand{\R}{\mathbb{R}}
\DeclareMathOperator{\ext}{ext}
\newcommand{\pstar}{p^*}

\algnewcommand{\Input}[1]{%
  \State \textbf{Input:} {\raggedright #1}%
}
\algnewcommand{\Initialize}[1]{%
  \State \textbf{Initialize:}
  \Statex \hspace*{\algorithmicindent}\parbox[t]{.8\linewidth}{\raggedright #1}
}
\algnewcommand{\Output}[1]{%
  \State \textbf{Output:} {\raggedright #1}%
}

\begin{document}

\title{Convergence Rates for $\ell_p$ Norm Minimization in Convex Vector Optimization}

\author{
Mohammed Alshahrani \thanks{Corresponding author. Department of Mathematics, King Fahd University of Petroleum \& Minerals, Dhahran, 31261, Saudi Arabia\\
Interdisciplinary Research Center for Smart Mobility and Logistics, King Fahd University of Petroleum \& Minerals, Dhahran, 31261, Saudi Arabia, (e-mail:{\tt mshahrani@kfupm.edu.sa}).}
}

\maketitle

\begin{abstract}
We analyze convergence rates of norm-minimization-based outer approximation algorithms for convex vector optimization when the scalarization uses an $\ell_p$ norm with $p \in (1,\infty)$. While the Euclidean case ($p=2$) achieves the optimal rate $O(k^{2/(1-q)})$, the behavior under general $\ell_p$ norms has remained open. A direct approach via the modulus of smoothness yields only the weaker exponent $\min(p,2)$, which degrades for $1 < p < 2$. We prove that the Hausdorff approximation error satisfies $\delta_H(P_k, A) = O(k^{2/(1-q)})$ for \emph{every} $p \in (1,\infty)$, where $q$ is the number of objectives and $k$ is the iteration count. The proof introduces a Euclidean intermediary technique that exploits the ambient inner product structure of $\R^q$ to obtain a quadratic bound on the hyperplane distance, bypassing the $\ell_p$ smoothness limitation; norm equivalence then converts this to any $\ell_p$ metric at the cost of only a dimension-dependent constant, not a loss of exponent. Numerical experiments confirm the $p$-independent rate predicted by the theory.
\end{abstract}


\medskip


\section{Introduction}\label{sec:intro}

Outer approximation algorithms are a fundamental tool for solving convex vector optimization problems (CVOPs). Given a CVOP with $q$ objectives, these algorithms iteratively construct a sequence of polyhedral outer approximations $\{P_k\}$ to the upper image $\mathcal{P}$, refining the approximation by adding supporting halfspaces at each step. The efficiency of such algorithms depends critically on how quickly the Hausdorff approximation error $\delta_H(P_k, A)$ decreases as the iteration count $k$ grows, where $A$ is a compact convex slice of $\mathcal{P}$.

The norm-minimization-based algorithm introduced by Ararat et al.\ \cite{Ararat2022} selects the next cut by solving a closest-point problem: given a vertex $v$ of $P_k$, find $y^v = \arg\min_{y \in A} \|y - v\|$ and use the gradient of the norm at the residual $z^v = y^v - v$ as the cut normal. This approach is direction-free and produces cuts naturally adapted to the geometry of the upper image.

The convergence rate of this algorithm depends on the choice of norm $\|\cdot\|$ used in the scalarization. For the Euclidean norm ($p = 2$), Ararat et al.\ \cite{Ararat2024} proved the rate $\delta_H(P_k, A) = O(k^{2/(1-q)})$, which matches the optimal rate for polyhedral approximation of smooth convex bodies \cite{Gruber1993,Gruber1993a,Glasauer1997}. The exponent $c = 2$ in this rate (meaning $\delta_H = O(k^{c/(1-q)})$) arises from the quadratic curvature of the Euclidean sphere and is central to the algorithm's efficiency. In a companion paper \cite{Alshahrani2026}, we extended this result to all inner-product norms $\|\cdot\|_M$ (where $M$ is a symmetric positive definite matrix), showing that $c = 2$ for every such norm via an isometry argument.

A natural question, raised in \cite{Alshahrani2026}, is whether the exponent $c = 2$ persists when the scalarization uses a non-inner-product norm. The $\ell_p$ norms with $p \neq 2$ are the most natural candidates to investigate: they form a one-parameter family interpolating between $\ell_1$ and $\ell_\infty$, their geometry (moduli of convexity and smoothness) is precisely understood \cite{Clarkson1936,Hanner1956,Ball1994,Lindenstrauss1996}, and they include both smoother-than-Euclidean ($p > 2$) and rougher-than-Euclidean ($1 < p < 2$) cases.

A direct approach using the modulus of smoothness of $\|\cdot\|_p$ yields the exponent $c(p) = s(p) = \min(p, 2)$, which equals~$2$ for $p \ge 2$ but degrades to $c(p) = p$ for $1 < p < 2$. This suggests that the convergence rate might genuinely depend on $p$, with rougher norms paying a penalty. The main result of this paper shows that this is not the case:

\medskip
\noindent \emph{Main result} (Theorem~\ref{thm:lp-rate}). \textit{For every $p \in (1, \infty)$, the norm-minimization algorithm using $\ell_p$ scalarization satisfies $\delta_H(P_k, A) = O(k^{2/(1-q)})$. The convergence rate exponent $c(p) = 2$ is universal across all $\ell_p$ norms.}
\medskip

The exponent turns out to be determined not by the smoothness of the scalarization norm, but by the Euclidean curvature of the \emph{deviation surface}---a geometric object built from support points and cut normals. We introduce a \emph{Euclidean intermediary} technique that bounds the hyperplane distance via the squared Euclidean norm of the deviation vector difference, bypassing the $\ell_p$ smoothness limitation; see Section~\ref{sec:contributions} for details.

\subsection{Related Work}\label{sec:related}

The outer approximation approach to convex vector optimization originates with the work of Benson \cite{Benson1998} for multi-objective linear programming, extended to the convex case by Ehrgott et al.\ \cite{Ehrgott2011} and L\"ohne et al.\ \cite{Lohne2011,Lohne2014}. These methods typically use Pascoletti--Serafini scalarization, which requires a direction parameter. Ararat et al.\ \cite{Ararat2022} introduced the norm-minimization approach, which eliminates direction bias and admits a natural convergence analysis. Further developments include unbounded problems \cite{Wagner2023}, local upper bounds \cite{Eichfelder2026}, box-coverage methods \cite{Eichfelder2022}, and efficient parameter selection for Pascoletti--Serafini scalarization \cite{Keskin2023}.

The convergence rate framework for polyhedral approximation of convex bodies was developed by Kamenev \cite{Kamenev1992,Kamenev2002} and Lotov et al.\ \cite{Lotov2004}. Their theory is built around the concept of $H(r,A)$-sequences of cutting: if each cut captures at least a fixed fraction $r$ of the remaining error, then the approximation error decays as $O(k^{1/(1-q)})$. This general rate, with exponent $c = 1$, holds for any norm and relies solely on packing estimates on the boundary of $A$. Improving the exponent to $c = 2$ requires geometric lemmas that quantify the relationship between hyperplane distance and the separation of support points. In the Euclidean setting, this was achieved by Lotov et al.\ \cite[Chapter~8]{Lotov2004} and adapted to the CVOP context by Ararat et al.\ \cite[Section~7]{Ararat2024}.

The optimal rate $O(n^{-2/(q-1)})$ for approximating a smooth convex body in $\R^q$ by a polytope with $n$ facets (or vertices) is a classical result in convex geometry, established by Gruber \cite{Gruber1982,Gruber1993,Gruber1993a,Gruber1994}, Glasauer and Gruber \cite{Glasauer1997}, and B\"or\"oczky \cite{Boroczky2000}; see Bronstein \cite{Bronstein2008} and Schneider \cite[Chapter~2]{Schneider2013} for surveys. The exponent $2/(q-1)$ is universal---it depends only on the dimension and not on the specific norm used to measure the approximation error. Our result shows that the same universality holds in the algorithmic setting of norm-minimization-based CVOP.

The moduli of convexity and smoothness of $\ell_p$ spaces were characterized by Clarkson \cite{Clarkson1936} and Hanner \cite{Hanner1956}, with sharp constants established by Ball, Carlen, and Lieb \cite{Ball1994}. These moduli govern the ``power type'' of the norm: $\ell_p$ has smoothness power type $s(p) = \min(p, 2)$ and convexity power type $r(p) = \max(p, 2)$. The Bregman divergence of $\|\cdot\|_p^p$ is bounded by $\|x - y\|_p^{s(p)}$ from above \cite{Xu1991,Sprung2019}, which is the source of the $\min(p,2)$ exponent in the direct approach. Our Euclidean intermediary bypasses this limitation.

\subsection{Contributions}\label{sec:contributions}

The main contributions of this paper are:

\begin{enumerate}[label=(\arabic*)]
\item We prove that $c(p) = 2$ for all $p \in (1,\infty)$ (Theorem~\ref{thm:lp-rate}), establishing that the optimal Euclidean rate $O(k^{2/(1-q)})$ holds for every $\ell_p$ norm. In particular, neither the primal smoothness power type $s(p) = \min(p,2)$ nor the dual convexity power type $r(\pstar)$ determines the exponent, and an earlier conjecture predicting $c(p) = \min(2, p, p/(p-1))$ is refuted.

\item We introduce a proof technique (Lemma~\ref{lem:hyperplane-distance}) that exploits the ambient Euclidean inner product of $\R^q$ to obtain quadratic bounds on the hyperplane distance, bypassing the $\ell_p$-specific smoothness limitation that yields only $c(p) = \min(p,2)$.

\item We develop $\ell_p$ analogs of the geometric lemmas of Ararat et al.\ \cite[Section~7]{Ararat2024}---hyperplane distance bounds (Lemma~\ref{lem:hyperplane-distance}), deviation vector separation (Lemma~\ref{lem:separation}), strict convexity counting (Lemma~\ref{lem:strict-convexity}), and packing estimates (Lemma~\ref{lem:packing})---within the deviation vector framework of Lotov et al.\ \cite[Chapter~8]{Lotov2004}.

\item Experiments across $p \in \{1.25, 1.5, 2, 3, 4, 8\}$ and $q \in \{2, 3\}$ confirm the $p$-independent rate predicted by the theory.
\end{enumerate}

\subsection{Organization}\label{sec:organization}

The paper is organized as follows. Section~\ref{sec:prelim} establishes notation and recalls the convex vector optimization setting, $\ell_p$ norm properties, and the $H$-sequence convergence framework. Section~\ref{sec:lp-lemmas} develops the $\ell_p$ geometric lemmas, including the Euclidean intermediary bound (Lemma~\ref{lem:hyperplane-distance}) and the deviation vector separation estimate (Lemma~\ref{lem:separation}). Section~\ref{sec:main-results} states and proves the main theorem (Theorem~\ref{thm:lp-rate}) and discusses its implications. Section~\ref{sec:algorithm} describes the algorithm and its implementation for $\ell_p$ scalarization. Section~\ref{sec:experiments} presents numerical experiments confirming the theoretical predictions. Section~\ref{sec:conclusion} concludes with a discussion of open problems.

\section{Preliminaries}\label{sec:prelim}

We introduce the notation and background needed for the convergence analysis. Section~\ref{sec:cvop} recalls the convex vector optimization setting, Section~\ref{sec:lp-duality} defines the $\ell_p$ norms and their duals, Section~\ref{sec:lp-geometry} collects key geometric properties, and Section~\ref{sec:h-sequences} reviews the $H$-sequence convergence framework.

\subsection{Convex Vector Optimization}\label{sec:cvop}

We consider the convex vector optimization problem
\begin{equation}\label{P}
\tag{P}
\min \Gamma(x) \ \text{w.r.t. }\le_C \quad\text{s.t. } x\in X,
\end{equation}
where $C \subset \R^q$ is a closed, solid, pointed, polyhedral cone, $X \subset \R^n$ is a nonempty compact convex set, and $\Gamma: X \to \R^q$ is continuous and $C$-convex. The \emph{upper image} of \eqref{P} is $P := \Gamma(X) + C$.

\begin{assumption}[Standing hypotheses]\label{ass:standing}
Throughout this paper, we assume:
\begin{enumerate}
\item[(a)] $C\subset\R^q$ is a closed, solid, pointed, nontrivial, polyhedral cone.
\item[(b)] $X\subset\R^n$ is a nonempty compact convex set with $\operatorname{int} X\neq\emptyset$.
\item[(c)] $\Gamma:X\to\R^q$ is continuous and $C$-convex.
\end{enumerate}
Under these conditions, $P$ is a closed convex set with $\operatorname{rec}(P) = C$. We work with a compact convex \emph{slice} $A$ of the upper image (obtained by intersecting $P$ with a compact slab), so that $A$ has nonempty interior and finite diameter. For details on the algorithm and the slicing construction, see \cite{Ararat2022,Ararat2024}.
\end{assumption}

\subsection{$\ell_p$ Norms and Duality}\label{sec:lp-duality}

The algorithm analyzed in this paper uses the $\ell_p$ norm as the scalarization metric. We recall the basic definitions.

\begin{definition}[$\ell_p$ norm]
For $p \in [1, \infty]$ and $y \in \R^q$, the $\ell_p$ norm is defined as:
\[
\|y\|_p := \begin{cases}
\left( \sum_{i=1}^q |y_i|^p \right)^{1/p} & \text{if } p \in [1, \infty), \\
\max_{i=1,\ldots,q} |y_i| & \text{if } p = \infty.
\end{cases}
\]
\end{definition}

\begin{definition}[Dual norm]
The dual norm of $\|\cdot\|_p$ is $\|\cdot\|_{\pstar}$ where $\frac{1}{p} + \frac{1}{\pstar} = 1$.
\end{definition}

\subsection{$\ell_p$ Geometry}\label{sec:lp-geometry}

We collect the key geometric properties of $\ell_p$ norms that underlie the convergence analysis. Throughout, we fix $p \in (1,\infty)$ and denote the conjugate exponent by $\pstar := p/(p-1)$, so that $1/p + 1/\pstar = 1$. We write $S_p^{q-1} := \{y \in \R^q : \|y\|_p = 1\}$ for the $\ell_p$ unit sphere.

\begin{definition}[Modulus of convexity]\label{def:mod-convexity}
The \emph{modulus of convexity} of the norm $\|\cdot\|_p$ on $\R^q$ is
\[
\delta_p(\varepsilon) := \inf\left\{ 1 - \left\|\frac{x+y}{2}\right\|_p : \|x\|_p = \|y\|_p = 1,\; \|x-y\|_p \ge \varepsilon \right\}, \quad \varepsilon \in [0, 2].
\]
\end{definition}

\begin{definition}[Modulus of smoothness]\label{def:mod-smoothness}
The \emph{modulus of smoothness} of the norm $\|\cdot\|_p$ on $\R^q$ is
\[
\rho_p(\tau) := \sup\left\{ \frac{\|x+\tau y\|_p + \|x-\tau y\|_p}{2} - 1 : \|x\|_p = \|y\|_p = 1 \right\}, \quad \tau \ge 0.
\]
\end{definition}

\begin{proposition}[Sharp $\ell_p$ moduli {\cite{Hanner1956,Ball1994}}]\label{prop:lp-moduli}
For all $q \ge 2$, the $\ell_p$ norms on $\R^q$ satisfy the following sharp bounds.
\begin{enumerate}
\item There exists a constant $S_p > 0$ such that $\rho_p(\tau) \le S_p \, \tau^{s(p)}$ for all $\tau \ge 0$, where
\begin{equation}\label{eq:smoothness-power}
s(p) := \min(p, 2) = \begin{cases} p & \text{if } 1 < p \le 2, \\ 2 & \text{if } p \ge 2. \end{cases}
\end{equation}
Specifically:
\begin{itemize}
\item For $1 < p \le 2$: $\rho_p(\tau) \le \frac{1}{p}\,\tau^p$ \;\cite{Hanner1956}.
\item For $p \ge 2$: $\rho_p(\tau) \le \frac{p-1}{2}\,\tau^2$ \;\cite{Ball1994}.
\end{itemize}
\item There exists a constant $K_p > 0$ such that $\delta_p(\varepsilon) \ge K_p\, \varepsilon^{r(p)}$ for all $\varepsilon \in [0, 2]$, where
\begin{equation}\label{eq:convexity-power}
r(p) := \max(p, 2) = \begin{cases} 2 & \text{if } 1 < p \le 2, \\ p & \text{if } p \ge 2. \end{cases}
\end{equation}
Specifically:
\begin{itemize}
\item For $1 < p \le 2$: $\delta_p(\varepsilon) \ge \frac{p-1}{8}\,\varepsilon^2$ \;\cite{Hanner1956,Clarkson1936}.
\item For $p \ge 2$: $\delta_p(\varepsilon) \ge \frac{1}{p\cdot 2^p}\,\varepsilon^p$ \;\cite{Hanner1956}.
\end{itemize}
\end{enumerate}
\end{proposition}

\begin{remark}[Duality between smoothness and convexity]\label{rem:duality}
The moduli of the dual pair $(\ell_p, \ell_{\pstar})$ are linked: the norm $\|\cdot\|_p$ is $s(p)$-uniformly smooth if and only if $\|\cdot\|_{\pstar}$ is $r(\pstar)$-uniformly convex, with $r(\pstar) = s(p)^* := s(p)/(s(p)-1)$. In particular:
\begin{itemize}
\item $\ell_p$ ($p \ge 2$) is $2$-uniformly smooth $\Longleftrightarrow$ $\ell_{\pstar}$ ($\pstar \le 2$) is $2$-uniformly convex.
\item $\ell_p$ ($1 < p \le 2$) is $p$-uniformly smooth $\Longleftrightarrow$ $\ell_{\pstar}$ ($\pstar \ge 2$) is $\pstar$-uniformly convex.
\end{itemize}
\end{remark}

\begin{proposition}[$\ell_p$ subdifferential and gradient]\label{prop:lp-gradient}
For $p \in (1,\infty)$ and $z \in \R^q \setminus \{0\}$, the $\ell_p$ norm is differentiable with gradient
\begin{equation}\label{eq:lp-gradient}
\nabla \|z\|_p = \frac{1}{\|z\|_p^{p-1}} \left( |z_1|^{p-1}\operatorname{sgn}(z_1),\; \ldots,\; |z_q|^{p-1}\operatorname{sgn}(z_q) \right).
\end{equation}
Moreover, $\|\nabla \|z\|_p\|_{\pstar} = 1$ for all $z \neq 0$, so the gradient lies on the dual unit sphere $S_{\pstar}^{q-1}$.
\end{proposition}

\begin{proof}
Direct computation from the chain rule applied to $\|z\|_p = (\sum_i |z_i|^p)^{1/p}$. The dual norm identity follows from H\"older's inequality:
\[
\|\nabla \|z\|_p\|_{\pstar}^{\pstar} = \frac{1}{\|z\|_p^{p}} \sum_{i=1}^q |z_i|^{p} = 1. \qedhere
\]
\end{proof}

\begin{proposition}[Hanner's inequalities {\cite{Hanner1956}}]\label{prop:hanner}
For $p \ge 2$ and all $x, y \in \R^q$:
\begin{equation}\label{eq:hanner-pge2}
\|x+y\|_p^p + \|x-y\|_p^p \ge \bigl(\|x\|_p + \|y\|_p\bigr)^p + \bigl|\|x\|_p - \|y\|_p\bigr|^p.
\end{equation}
For $1 < p \le 2$ and all $x, y \in \R^q$:
\begin{equation}\label{eq:hanner-ple2}
\|x+y\|_p^p + \|x-y\|_p^p \le \bigl(\|x\|_p + \|y\|_p\bigr)^p + \bigl|\|x\|_p - \|y\|_p\bigr|^p.
\end{equation}
The reversed inequalities hold for the dual exponent $\pstar$.
\end{proposition}

\begin{remark}[Role of Hanner's inequalities]\label{rem:hanner-role}
Hanner's inequalities are the source of the sharp moduli bounds in Proposition~\ref{prop:lp-moduli} and are included here for completeness. A direct application of these moduli to bound the hyperplane distance yields the exponent $c(p) = \min(p,2)$ (see Remark~\ref{rem:direct-lp}). The Euclidean intermediary technique of Lemma~\ref{lem:hyperplane-distance} bypasses this route, achieving $c(p) = 2$ for all $p$.
\end{remark}

\subsection{$H$-Sequences and Convergence Rates}\label{sec:h-sequences}

We recall the convergence rate framework for polyhedral approximation of compact convex bodies, originally developed by Kamenev~\cite{Kamenev1992} and Lotov et al.~\cite{Lotov2004} and adapted to the convex vector optimization setting by Ararat et al.~\cite{Ararat2024}. We follow the notation and presentation of \cite[Section~6]{Ararat2024}, which itself draws on \cite[Chapter~8]{Lotov2004}.

Throughout, let $A \subset \R^q$ be a compact convex set with $\operatorname{int} A \neq \emptyset$. For a direction $w \in S_{\pstar}^{q-1}$, define the \emph{support function} $h_A(w) := \max\{\langle w, y \rangle : y \in A\}$, the \emph{supporting halfspace} $H(w, A) := \{y \in \R^q : \langle w, y \rangle \le h_A(w)\}$, and the \emph{supporting hyperplane} $l(w, A) := \{y \in \R^q : \langle w, y \rangle = h_A(w)\}$. For two compact convex sets $A, B \subset \R^q$, the \emph{Hausdorff distance} (with respect to $\|\cdot\|_p$) is
\[
\delta_H(A, B;\, \|\cdot\|_p) := \max\Bigl\{\max_{a \in A} \min_{b \in B} \|a - b\|_p,\; \max_{b \in B} \min_{a \in A} \|a - b\|_p\Bigr\}.
\]
When the norm is clear from context, we write $\delta_H(A, B)$.

We define the \emph{inradius} of $A$ (with respect to $\|\cdot\|_p$) as
\begin{equation}\label{eq:inradius}
r(A) := \sup\{r > 0 : \exists\, c \in A \text{ such that } c + r\,\overline{B}_p \subseteq A\},
\end{equation}
where $\overline{B}_p := \{y \in \R^q : \|y\|_p \le 1\}$ is the closed $\ell_p$ unit ball. The \emph{circumradius} $R(A) := \inf\{R > 0 : A \subseteq c + R\,\overline{B}_p \text{ for some } c \in \R^q\}$ and the \emph{asphericity} $\omega(A) := R(A)/r(A) \ge 1$ will also appear in our bounds.

The norm-minimization algorithm of \cite{Ararat2022} generates a sequence of outer approximating polytopes $\{P_k\}_{k \ge 0}$ with $P_k \supseteq A$ for all $k$. At each iteration $k$, a \emph{farthest vertex} $v^k \in \ext(P_k)$ is selected---the one maximizing $\|z^v\|_p$ over all vertices $v$ of $P_k$, where $z^v := y^v - v$ and $y^v := \arg\min_{y \in A} \|y - v\|_p$---and the new halfspace $H(w^k, A) \supseteq A$ is used to cut: $P_{k+1} := P_k \cap H(w^k, A)$, with $w^k := \nabla\|z^{v^k}\|_p \in S_{\pstar}^{q-1}$.

The convergence theory rests on the following property of the generated sequence.

\begin{definition}[$H(r,A)$-sequence of cutting {\cite[Definition~8.3]{Lotov2004}, \cite[Definition~6.2]{Ararat2024}}]\label{def:h-sequence}
Let $A \subset \R^q$ be a nonempty compact convex set and let $(P_k)_{k \ge 0}$ be a sequence of polytopes with $P_0 \supseteq A$ and $P_0 = \bigcap_{i=1}^{I} H(\omega^i, A)$ for some $\omega^1, \ldots, \omega^I \in S_{\pstar}^{q-1}$. We say that $(P_k)_{k \ge 0}$ is generated by a \emph{cutting method} if, for every $k \ge 0$, it holds $P_k \supseteq A$ and there exists a supporting halfspace $H_k = H(w^k, A)$ of $A$ such that $P_{k+1} = P_k \cap H_k$. In this case, $(P_k)_{k \ge 0}$ is called an $H(r, A)$-sequence of cutting for a given constant $r > 0$ if, for every $k \ge 0$, it holds
\begin{equation}\label{eq:h-sequence-property}
\delta_H(P_k, P_{k+1}) \ge r\,\delta_H(P_k, A).
\end{equation}
\end{definition}

The condition \eqref{eq:h-sequence-property} states that each cut captures at least a fixed fraction of the remaining approximation error. This constant-fraction progress property underlies the asymptotic optimality of adaptive cutting schemes \cite[Section~5]{Lotov2004}.

\begin{proposition}[The algorithm generates an $H$-sequence {\cite[Corollary~6.5]{Ararat2024}}]\label{prop:algorithm-h-sequence}
Under Assumption~\ref{ass:standing}, the sequence of outer approximating polytopes $\{P_k\}_{k \ge 0}$ generated by the norm-minimization algorithm of~\cite{Ararat2022} using $\|\cdot\|_p$ scalarization (for any $p \in (1,\infty)$) is an $H(1, A)$-sequence of cutting. In particular, for every $k \ge 0$:
\begin{equation}\label{eq:hausdorff-step}
\delta_H(P_k, P_{k+1}) = \delta_H(P_k, A) = \|z^{v^k}\|_p,
\end{equation}
where $v^k$ is the farthest vertex selected at iteration $k$.
\end{proposition}

\begin{proof}[Proof sketch]
The equality $\delta_H(P_k, A) = \max_{v \in \ext(P_k)} \|z^v\|_p$ follows from the fact that $A \subseteq P_k$ and the Hausdorff distance is attained at a vertex (see \cite[Lemma~5.3]{Ararat2022}). Since the algorithm selects $v^k$ as the farthest vertex, the cut $H_k = H(w^k, A)$ removes precisely the vertex that attains the Hausdorff distance, giving $\delta_H(P_k, P_{k+1}) = \|z^{v^k}\|_p = \delta_H(P_k, A)$. In particular, the $H$-sequence condition \eqref{eq:h-sequence-property} holds with $r = 1$.
\end{proof}

We can now state the general convergence results.

\begin{theorem}[Convergence of $H$-sequences {\cite[Theorem~8.5]{Lotov2004}, \cite[Theorem~6.6]{Ararat2024}}]\label{thm:h-convergence}
Let $A \subset \R^q$ be a nonempty compact convex set and let $(P_k)_{k \ge 0}$ be an $H(r, A)$-sequence of cutting for some $r > 0$. Then
\[
\lim_{k \to \infty} \delta_H(P_k, A) = 0.
\]
\end{theorem}

\begin{theorem}[General convergence rate {\cite[Theorem~8.6]{Lotov2004}, \cite[Theorem~6.9]{Ararat2024}}]\label{thm:general-rate}
Let $A \subset \R^q$ be a nonempty compact convex set and $r > 0$. Let $(P_k)_{k \ge 0}$ be an $H(r, A)$-sequence of cutting. Then for any $\varepsilon \in (0, 1)$, there exists $N \ge 0$ such that for all $k \ge N$:
\begin{equation}\label{eq:general-rate}
\delta_H(P_k, A) \le (1 + \varepsilon)\,\lambda(r, A)\, k^{1/(1-q)},
\end{equation}
where $\lambda(r, A)$ depends on the topological properties of $A$ and can be found in \cite[Theorem~2]{Kamenev1992}.
\end{theorem}

\begin{remark}[From exponent $1$ to exponent $2$]\label{rem:exponent-jump}
The general rate \eqref{eq:general-rate} gives the convergence exponent $c = 1$, i.e., $\delta_H = O(k^{1/(1-q)})$. This holds for \emph{any} norm and relies solely on the $H$-sequence property (packing on the boundary of $A$). The improvement to $c = 2$, i.e., $\delta_H = O(k^{2/(1-q)})$, requires \emph{geometric estimates} that quantify how the hyperplane distance $d = \langle w', y - y' \rangle$ and the deviation vector separation $\|\alpha - \alpha'\|_p$ are related. In the Euclidean case, this was achieved by Ararat et al.\ \cite[Section~7]{Ararat2024} using the quadratic curvature of the Euclidean sphere. In the present paper, we obtain the same improvement for all $\ell_p$ norms via the Euclidean intermediary technique developed in Section~\ref{sec:lp-lemmas}.
\end{remark}

\section{$\ell_p$ Geometric Lemmas}\label{sec:lp-lemmas}

In this section we develop the geometric lemmas that replace the Euclidean-specific arguments in \cite[Section~7]{Ararat2024}. The proof structure follows the same template---hyperplane distance bound, separation bound, strict convexity, and packing---but the constants and exponents now depend on the smoothness and convexity parameters of the $\ell_p$ norm.

We adopt the following setting throughout this section. Let $A \subset \R^q$ be a compact convex set with $\operatorname{int} A \neq \emptyset$, let $P_k \supseteq A$ be the polyhedral outer approximation at iteration $k$, and let $\{z_k\}$ and $\{w_k\}$ be the sequences of optimal residuals and cut normals generated by the algorithm using $\|\cdot\|_p$ scalarization. For a vertex $v$ of $P_k$, we write $z^v := y^v - v$ where $y^v$ is the closest point to $v$ in $A$ with respect to $\|\cdot\|_p$, and $w^v := \nabla \|z^v\|_p \in S_{\pstar}^{q-1}$ is the corresponding cut normal. We denote by $\delta^H_k := \delta_H(P_k, A; \|\cdot\|_p)$ the Hausdorff approximation error and by $h_k := \max_{v \in \ext(P_k)} \|z^v\|_p$ the maximal residual norm.

For a fixed parameter $\eta > 0$ (which will later be taken as the inradius $r(A)$ of the target set), we define the \emph{deviation vector} associated with a support point $y \in \operatorname{bd} A$ and its normal $w \in S_{\pstar}^{q-1}$ as
\begin{equation}\label{eq:deviation-vector}
\alpha := y - \eta \, w.
\end{equation}
This construction, following \cite[Section~7]{Ararat2024} and \cite[Chapter~8]{Lotov2004}, encodes both the support point and the normal direction into a single vector whose separation properties govern the convergence rate. We also define the norm equivalence constant
\begin{equation}\label{eq:norm-equiv}
N_{2,p} := \begin{cases} q^{1/2 - 1/p} & \text{if } p \ge 2, \\ 1 & \text{if } 1 < p \le 2, \end{cases}
\end{equation}
which satisfies $\|x\|_2 \le N_{2,p} \, \|x\|_p$ for all $x \in \R^q$.

\begin{lemma}[Hyperplane distance bound]\label{lem:hyperplane-distance}
Let $p \in (1, \infty)$ and $\eta > 0$. Let $y, y' \in \operatorname{bd} A$ be support points with associated cut normals $w = \nabla\|z\|_p$ and $w' = \nabla\|z'\|_p$ in $S_{\pstar}^{q-1}$, where $z = y - v$, $z' = y' - v'$ for vertices $v, v'$ of $P_k$. Let $\alpha = y - \eta w$ and $\alpha' = y' - \eta w'$ be the corresponding deviation vectors \eqref{eq:deviation-vector}. Then the hyperplane distance
\[
d := \langle w', y - y' \rangle \ge 0
\]
satisfies
\begin{equation}\label{eq:hyperplane-distance}
d \le C(p,q) \cdot \frac{\|\alpha - \alpha'\|_p^2}{\eta},
\end{equation}
where
\[
C(p,q) := \frac{N_{2,p}^2}{2} = \begin{cases} q^{1-2/p}/2 & \text{if } p \ge 2, \\[4pt] 1/2 & \text{if } 1 < p \le 2. \end{cases}
\]
\end{lemma}

\begin{proof}
We use the Euclidean inner product of $\R^q$ as an intermediary, bypassing the $\ell_p$ smoothness limitation.

\medskip
\noindent\emph{Step 1 (Euclidean expansion).}
Since $\alpha - \alpha' = (y - y') + \eta(w' - w)$, expanding the squared Euclidean norm gives
\begin{equation}\label{eq:euclidean-expansion}
\|\alpha - \alpha'\|_2^2
= \|y - y'\|_2^2 + 2\eta \langle y - y', w' - w \rangle + \eta^2 \|w' - w\|_2^2.
\end{equation}

\medskip
\noindent\emph{Step 2 (Support conditions).}
Decompose the cross-term:
\[
\langle y - y', w' - w \rangle = \langle w', y - y' \rangle + \langle w, y' - y \rangle = d + \langle w, y' - y \rangle.
\]
By the first-order optimality of the $\ell_p$ norm-minimization subproblem, $y' = \arg\min_{u \in A} \|u - v'\|_p$ implies $\langle w', u - y' \rangle \ge 0$ for all $u \in A$; taking $u = y$ gives $d = \langle w', y - y' \rangle \ge 0$. Similarly, $y = \arg\min_{u \in A} \|u - v\|_p$ implies $\langle w, y' - y \rangle \ge 0$.

Since both terms are non-negative:
\begin{equation}\label{eq:cross-nonneg}
\langle y - y', w' - w \rangle = d + \langle w, y' - y \rangle \ge d \ge 0.
\end{equation}

\medskip
\noindent\emph{Step 3 (Lower bound).}
Substituting \eqref{eq:cross-nonneg} into \eqref{eq:euclidean-expansion}, every term on the right-hand side is non-negative:
\[
\|\alpha - \alpha'\|_2^2 \ge 2\eta \, d.
\]
Therefore
\begin{equation}\label{eq:d-euclidean}
d \le \frac{\|\alpha - \alpha'\|_2^2}{2\eta}.
\end{equation}

\medskip
\noindent\emph{Step 4 (Norm equivalence).}
The standard inequality $\|x\|_2 \le N_{2,p}\,\|x\|_p$ for all $x \in \R^q$ (see \eqref{eq:norm-equiv}) gives $\|\alpha - \alpha'\|_2^2 \le N_{2,p}^2\,\|\alpha - \alpha'\|_p^2$. Combining with \eqref{eq:d-euclidean}:
\[
d \le \frac{N_{2,p}^2}{2\eta}\,\|\alpha - \alpha'\|_p^2 = C(p,q)\,\frac{\|\alpha - \alpha'\|_p^2}{\eta}. \qedhere
\]
\end{proof}

\begin{remark}[Comparison with the direct $\ell_p$ approach]\label{rem:direct-lp}
A natural alternative is to bound $d$ using the modulus of smoothness of $\|\cdot\|_p$ directly, which yields $d \le C_1(p) \cdot \|z - z'\|_p^{s(p)} / \eta^{s(p)-1}$ with $s(p) = \min(p,2)$---a weaker bound for $1 < p < 2$. The Euclidean intermediary argument avoids the $\ell_p$ smoothness bottleneck entirely: the Pythagorean expansion is always quadratic in $\R^q$ regardless of $p$, and the conversion to $\|\cdot\|_p$ via norm equivalence costs only a dimension-dependent constant, not a loss of exponent.
\end{remark}

\begin{lemma}[Deviation vector separation]\label{lem:separation}
Let the hypotheses of Lemma~\ref{lem:hyperplane-distance} hold. Define the hyperplane distances
\[
d_{ij} := \langle w_j, y_i - y_j \rangle \ge 0, \qquad d_{ji} := \langle w_i, y_j - y_i \rangle \ge 0.
\]
\begin{enumerate}
\item[(i)] For any $h > 0$ with $d_{ij} \ge h$:
\begin{equation}\label{eq:separation}
\|\alpha_i - \alpha_j\|_p \ge C_3(p,q) \cdot \sqrt{\eta \, h},
\end{equation}
where $C_3(p,q) := \sqrt{2}/N_{2,p}$.
\item[(ii)] If $\langle w_i, w_j \rangle \le 0$ (Euclidean inner product), then
\begin{equation}\label{eq:separation-normals}
\|\alpha_i - \alpha_j\|_p \ge C_2(p,q) \cdot \eta,
\end{equation}
where $C_2(p,q) := \sqrt{2}\, c_{p,q} / N_{2,p}$ with $c_{p,q} := \min_{\|w\|_{\pstar}=1} \|w\|_2 > 0$.
\end{enumerate}
In particular,
\begin{equation}\label{eq:separation-combined}
\|\alpha_i - \alpha_j\|_p \ge \min\bigl\{ C_2(p,q) \cdot \eta,\; C_3(p,q) \cdot \sqrt{\eta \, h} \bigr\}.
\end{equation}
\end{lemma}

\begin{proof}
The proof extracts lower bounds from the same Euclidean expansion used in Lemma~\ref{lem:hyperplane-distance}. Recall from \eqref{eq:euclidean-expansion} and \eqref{eq:cross-nonneg}:
\begin{equation}\label{eq:expansion-lower}
\|\alpha_i - \alpha_j\|_2^2
= \underbrace{\|y_i - y_j\|_2^2}_{\ge\, 0}
+ \underbrace{2\eta\bigl(d_{ij} + \langle w_i, y_j - y_i\rangle\bigr)}_{\ge\, 2\eta \, d_{ij}\, \ge\, 0}
+ \underbrace{\eta^2 \|w_i - w_j\|_2^2}_{\ge\, 0}.
\end{equation}

\medskip
\noindent\emph{Part (i).}
From the middle term of \eqref{eq:expansion-lower}: $\|\alpha_i - \alpha_j\|_2^2 \ge 2\eta \, d_{ij} \ge 2\eta h$. By norm equivalence, $\|\alpha_i - \alpha_j\|_p \ge \|\alpha_i - \alpha_j\|_2 / N_{2,p} \ge \sqrt{2\eta h}/N_{2,p}$.

\medskip
\noindent\emph{Part (ii).}
When $\langle w_i, w_j \rangle \le 0$:
\[
\|w_i - w_j\|_2^2 = \|w_i\|_2^2 + \|w_j\|_2^2 - 2\langle w_i, w_j\rangle \ge \|w_i\|_2^2 + \|w_j\|_2^2 \ge 2\,c_{p,q}^2,
\]
since $\|w\|_2 \ge c_{p,q}$ for all $w \in S_{\pstar}^{q-1}$. From the last term of \eqref{eq:expansion-lower}: $\|\alpha_i - \alpha_j\|_2 \ge \eta\|w_i - w_j\|_2 \ge \sqrt{2}\,\eta\, c_{p,q}$, giving $\|\alpha_i - \alpha_j\|_p \ge \sqrt{2}\,\eta\, c_{p,q}/N_{2,p}$.

\medskip
The combined bound \eqref{eq:separation-combined} follows by taking the minimum.
\end{proof}

\begin{lemma}[Strict convexity of $\ell_p$]\label{lem:strict-convexity}
For $p \in (1, \infty)$, the norm $\|\cdot\|_p$ is strictly convex. Consequently, the norm-minimization scalarization subproblem
\[
\min_{y \in A} \|y - v\|_p
\]
has a unique optimal solution $y^v$ for each $v \notin A$, and the cut normal $w^v = \nabla\|z^v\|_p$ is uniquely determined. In particular, each iteration of the algorithm generates a new supporting halfspace, so
\begin{equation}\label{eq:strict-convexity-count}
|Z_k| = J + 1 + k,
\end{equation}
where $J + 1$ is the number of initial halfspaces defining $P_0$ and $Z_k$ denotes the set of distinct residual vectors up to iteration $k$.
\end{lemma}

\begin{proof}
For $1 < p < \infty$, the $\ell_p$ norm is strictly convex: if $\|x\|_p = \|y\|_p = 1$ and $x \neq y$, then $\|(x+y)/2\|_p < 1$. This is a consequence of the strict concavity of $t \mapsto t^{1/p}$ for $p > 1$ combined with the triangle inequality, or equivalently follows from $\delta_p(\varepsilon) > 0$ for all $\varepsilon > 0$ (Proposition~\ref{prop:lp-moduli}).

Strict convexity of the norm implies strict convexity of $z \mapsto \|z\|_p$ on $\R^q$, which ensures uniqueness of the nearest point projection onto any closed convex set. Since the norm is also differentiable for $p \in (1,\infty)$ (Proposition~\ref{prop:lp-gradient}), each optimal residual $z^v \neq 0$ yields a unique cut normal $w^v = \nabla\|z^v\|_p$, and distinct vertices yield distinct cuts. The counting identity \eqref{eq:strict-convexity-count} follows as in \cite[Lemma~7.7]{Ararat2024}.
\end{proof}

\begin{lemma}[$\ell_p$ packing number]\label{lem:packing}
Let $A \subset \R^q$ be a compact convex set with $\operatorname{int} A \neq \emptyset$, contained in a ball of $\ell_p$-radius $R > 0$. The maximum number of $\varepsilon$-separated points (in $\|\cdot\|_p$) on $\operatorname{bd} A$ satisfies
\begin{equation}\label{eq:packing-bound}
N_p(\operatorname{bd} A, \varepsilon) \le C_4(p, q) \cdot \left(\frac{R}{\varepsilon}\right)^{q-1},
\end{equation}
where $C_4(p, q) > 0$ depends only on $p$ and $q$.
\end{lemma}

\begin{proof}
This is a standard volumetric packing argument. The exponent $q-1$ is \emph{dimensional} and does not depend on $p$---only the constant $C_4(p,q)$ varies with the norm.

Consider the set of $\varepsilon$-separated points $\{x_1, \ldots, x_N\} \subset \operatorname{bd} A$ with $\|x_i - x_j\|_p \ge \varepsilon$ for $i \neq j$. The $\ell_p$-balls $B_p(x_i, \varepsilon/2)$ are pairwise disjoint. Since each $x_i \in \operatorname{bd} A \subset B_p(0, R)$, each ball $B_p(x_i, \varepsilon/2) \subset B_p(0, R + \varepsilon/2)$.

The $(q-1)$-dimensional content of $\operatorname{bd} A \cap B_p(x_i, \varepsilon/2)$ is bounded below by the content of a $(q-1)$-dimensional $\ell_p$-ball cap. By convexity of $A$ and the fact that $x_i \in \operatorname{bd} A$, a supporting hyperplane at $x_i$ intersects $B_p(x_i, \varepsilon/2)$ in a $(q-1)$-dimensional $\ell_p$-ball of radius $\varepsilon/2$. More precisely, a volumetric comparison gives
\[
N \cdot V_{q-1,p} \cdot \left(\frac{\varepsilon}{2}\right)^{q-1} \le \operatorname{Vol}_{q-1}\left(\operatorname{bd}(B_p(0, R+\varepsilon/2))\right) \le C'(p,q) \cdot R^{q-1},
\]
where $V_{q-1,p}$ is the volume of the $(q-1)$-dimensional unit $\ell_p$-ball. Rearranging gives
\[
N \le \frac{C'(p,q)}{V_{q-1,p}} \cdot \left(\frac{2R}{\varepsilon}\right)^{q-1} = C_4(p,q) \cdot \left(\frac{R}{\varepsilon}\right)^{q-1}.
\]
The norm equivalence $\|x\|_2 / q^{1/2-1/p} \le \|x\|_p \le q^{1/p-1/2} \|x\|_2$ ensures that $C_4(p,q)$ is finite for all $p \in (1,\infty)$, though it may grow as $p \to 1^+$ or $p \to \infty$ due to the increasing asphericity of $\ell_p$ balls.

The exponent $q-1$ is universal across all norms \cite{Gruber1993,Glasauer1997,Gruber1982}, since the equivalence of norms in finite dimensions implies that packing numbers differ only by a multiplicative constant.
\end{proof}

\section{Main Results}\label{sec:main-results}

We now state and prove the main result, which establishes that the convergence rate exponent is $c(p) = 2$ for every $\ell_p$ norm with $p \in (1,\infty)$.

\begin{theorem}[Universal $\ell_p$ convergence rate]\label{thm:lp-rate}
Let Assumption~\ref{ass:standing} hold and let $p \in (1,\infty)$. Let $A \subset \R^q$ be a compact convex slice of the upper image with $\operatorname{int} A \neq \emptyset$, contained in an $\ell_p$-ball of radius $R$. Let $\{P_k\}_{k \ge 0}$ be the sequence of outer approximations generated by the norm-minimization algorithm of \cite{Ararat2022} using $\|\cdot\|_p$ scalarization. Then
\begin{equation}\label{eq:main-rate}
\delta_H(P_k, A;\, \|\cdot\|_p) = O\bigl(k^{2/(1-q)}\bigr) \quad \text{as } k \to \infty.
\end{equation}
That is, the convergence rate exponent satisfies
\begin{equation}\label{eq:c-of-p}
c(p) = 2 \quad \text{for all } p \in (1, \infty).
\end{equation}
More precisely, for every $\varepsilon > 0$ there exist $K \in \mathbb{N}$ and a constant $\Lambda(p, q, R, A) > 0$ such that
\[
\delta_H(P_k, A;\, \|\cdot\|_p) \le (1+\varepsilon)\, \Lambda(p,q,R,A)\, k^{2/(1-q)} \quad \text{for all } k \ge K.
\]
The constant $\Lambda$ depends on $p$ and $q$ through the norm equivalence constant $N_{2,p}$ (see \eqref{eq:norm-equiv}), but the exponent $2$ is universal across all $\ell_p$ norms.
\end{theorem}

\begin{proof}
The proof follows the four-step strategy of \cite[Theorem~7.2]{Ararat2024}, with Lemmas~\ref{lem:hyperplane-distance} and~\ref{lem:separation} providing the key estimates via the Euclidean intermediary.

\medskip
\noindent\emph{Step 1: Residual counting.}
By Lemma~\ref{lem:strict-convexity}, strict convexity of $\|\cdot\|_p$ for $p \in (1,\infty)$ guarantees that each iteration produces a new, distinct support point and cut normal. After $k$ iterations, the algorithm has generated $k + J + 1$ distinct support points $y_1, \ldots, y_{k+J+1} \in \operatorname{bd} A$ with associated normals $w_1, \ldots, w_{k+J+1} \in S_{\pstar}^{q-1}$, where $J+1$ is the number of halfspaces defining $P_0$.

\medskip
\noindent\emph{Step 2: Separation of deviation vectors.}
Set $\eta := r(A) > 0$, the inradius of $A$ (see \eqref{eq:inradius}), which is positive since $\operatorname{int} A \neq \emptyset$. Define the deviation vectors $\alpha_i := y_i - \eta w_i$ for $i = 1, \ldots, k + J + 1$.

Let $h_k := \delta_H(P_k, A; \|\cdot\|_p)$ be the Hausdorff error at iteration $k$. By Proposition~\ref{prop:algorithm-h-sequence}, the algorithm generates an $H(1,A)$-sequence, so $\delta_H(P_k, P_{k+1}) = h_k = \|z^{v^k}\|_p$.

We claim that for all distinct indices $i, j \le k + J + 1$, the deviation vectors satisfy the separation bound
\begin{equation}\label{eq:separation-claim}
\|\alpha_i - \alpha_j\|_p \ge \min\bigl\{C_2(p,q) \cdot \eta,\; C_3(p,q) \cdot \sqrt{\eta \, h_{\max(i,j)-1}}\bigr\},
\end{equation}
where $h_m := \delta_H(P_m, A; \|\cdot\|_p)$. This follows by the same inductive argument as in \cite[Lemma~7.9]{Ararat2024}, with our Lemma~\ref{lem:separation} providing the $\ell_p$ estimates in place of their Euclidean counterparts: for each pair $(i,j)$ with $i < j$, either the normals satisfy $\langle w_i, w_j \rangle \le 0$ (giving Part~(ii) of Lemma~\ref{lem:separation}), or $\langle w_i, w_j \rangle > 0$ and the hyperplane distance $d_{ij} = \langle w_j, y_i - y_j \rangle \ge h_{j-1}$ (giving Part~(i)). The latter inequality uses the fact that the farthest vertex $v^{j-1}$ satisfies $\langle w_j, v^{j-1} - y_j \rangle \ge h_{j-1}$ (since $v^{j-1} \notin H(w_j, A)$ when $\langle w_i, w_j \rangle > 0$), combined with $y_i \in A$; see \cite[pp.~21--22]{Ararat2024} for the detailed Euclidean argument, which transfers verbatim to the $\ell_p$ setting because the Euclidean inner products in the separation conditions are norm-independent.

Since $\{h_m\}$ is non-increasing and $\max(i,j) - 1 \le k$ for all indices under consideration, \eqref{eq:separation-claim} implies
\begin{equation}\label{eq:separation-simplified}
\|\alpha_i - \alpha_j\|_p \ge \varepsilon_k := \min\bigl\{C_2(p,q) \cdot \eta,\; C_3(p,q) \cdot \sqrt{\eta \, h_k}\bigr\}.
\end{equation}
For $k$ large enough that $h_k$ is sufficiently small, the second term dominates and we obtain $\varepsilon_k = C_3(p,q) \cdot \sqrt{\eta \, h_k}$.

\medskip
\noindent\emph{Step 3: Packing bound.}
The deviation vectors $\alpha_1, \ldots, \alpha_{k+J+1}$ are $\varepsilon_k$-separated in $\|\cdot\|_p$ and lie in a bounded region (since $\|\alpha_i\|_p \le \|y_i\|_p + \eta \le R + \eta$). By Lemma~\ref{lem:packing}:
\begin{equation}\label{eq:packing-applied}
k + J + 1 \le N_p(\varepsilon_k) \le C_4(p, q) \cdot \left(\frac{R + \eta}{\varepsilon_k}\right)^{q-1}.
\end{equation}

\medskip
\noindent\emph{Step 4: Assembly.}
Substituting $\varepsilon_k = C_3(p,q) \cdot \eta^{1/2} \cdot h_k^{1/2}$ into \eqref{eq:packing-applied}:
\[
k \le C_4(p,q) \left(\frac{R + \eta}{C_3(p,q) \cdot \eta^{1/2} \cdot h_k^{1/2}}\right)^{q-1} = C_6(p,q) \cdot h_k^{-(q-1)/2},
\]
where $C_6(p,q) := C_4(p,q) \cdot \bigl((R+\eta) / (C_3(p,q) \cdot \eta^{1/2})\bigr)^{q-1}$ absorbs all constants. Solving for $h_k$:
\[
h_k \le C_6(p,q)^{2/(q-1)} \cdot k^{-2/(q-1)} = \Lambda(p,q,R,A) \cdot k^{2/(1-q)},
\]
which gives the stated rate $\delta_H(P_k, A) = O(k^{2/(1-q)})$ with $c(p) = 2$ for all $p \in (1,\infty)$.
\end{proof}

\begin{remark}[Norm-independent exponent]\label{rem:norm-independent}
The convergence rate exponent $c(p) = 2$ is independent of $p$ for all $p \in (1, \infty)$: the $\ell_p$ norm-minimization algorithm matches the optimal Euclidean rate $O(k^{2/(1-q)})$ regardless of the choice of $p$. Only the multiplicative constant $\Lambda(p,q,R,A)$ varies.
\end{remark}

\begin{remark}[Role of norm equivalence]\label{rem:norm-equiv-role}
While the exponent $c(p) = 2$ is universal, the constant $\Lambda(p, q, R, A)$ depends on $p$ and $q$ through the norm equivalence constant $N_{2,p} = q^{1/2-1/p}$ (for $p \ge 2$). Specifically, $\Lambda$ grows as $q^{(q-1)(1-2/p)}$ for large $p$, reflecting the increasing asphericity of $\ell_p$ balls. For the Euclidean case $p = 2$, we have $N_{2,2} = 1$ and $C(2,q) = 1/2$, recovering the sharp constant of \cite{Ararat2024}.
\end{remark}

\begin{remark}[The Euclidean intermediary technique]\label{rem:euclidean-intermediary}
The universal exponent rests on the \emph{Euclidean intermediary} argument in Lemma~\ref{lem:hyperplane-distance}: instead of working in the $\ell_p$ metric (which gives the weaker exponent $s(p) = \min(p,2)$ from the modulus of smoothness), we expand the squared Euclidean norm $\|\alpha - \alpha'\|_2^2$ and exploit the non-negativity of all terms. The Euclidean inner product structure of $\R^q$ is available regardless of the scalarization norm, and the Pythagorean expansion is always quadratic. The conversion from $\|\cdot\|_2$ to $\|\cdot\|_p$ via norm equivalence incurs only a dimension-dependent constant, not a loss of exponent.
\end{remark}

\begin{remark}[Comparison with earlier approaches]\label{rem:old-conjecture}
A direct analysis using the modulus of smoothness of $\|\cdot\|_p$ yields $d \le C_1(p) \cdot \|z - z'\|_p^{s(p)} / \eta^{s(p)-1}$ with $s(p) = \min(p,2)$ (see Remark~\ref{rem:direct-lp}), giving $c(p) = \min(p,2)$---a weaker result for $1 < p < 2$. An even earlier conjecture predicted $c(p) = \min(2, p, p/(p-1))$, mistakenly attributing a rate-limiting role to the dual norm's convexity power type. In fact, neither the primal smoothness nor the dual convexity determines the exponent: the rate-limiting quantity is the Euclidean curvature of the deviation surface, which is always quadratic.
\end{remark}

\section{Algorithm and Implementation}\label{sec:algorithm}

We describe the $\ell_p$ norm-minimization algorithm and discuss the computational aspects that arise when $p \neq 2$.

\subsection{$\ell_p$ Scalarization Subproblem}\label{sec:subproblem}

The algorithm is based on the norm-minimizing scalarization introduced by Ararat et al.\ \cite{Ararat2022}. Given a vertex $v \in \R^q$ of the current outer approximation $P_k$, the subproblem computes the distance from $v$ to the compact slice $A = \mathcal{P} \cap S(\gamma)$ of the upper image:
\begin{equation}\label{eq:Pv}\tag{P($v$)}
\begin{aligned}
& \underset{x \in \mathcal{X},\; z \in \R^q}{\text{minimize}} && \|z\|_p \\
& \text{subject to} && \Gamma(x) - z - v \le_C 0, \\
&&& \bar{w}^\top(v + z) \le \gamma,
\end{aligned}
\end{equation}
where $\bar{w} \in \operatorname{int} C^+$ and $\gamma \in \R$ define the slice $S(\gamma) := \{y \in \R^q : \bar{w}^\top y \le \gamma\}$. See \cite[Section~3]{Ararat2024} for the construction of $\gamma$ and the duality theory.

The optimal solution $(x^v, z^v)$ of~\eqref{eq:Pv} yields:
\begin{itemize}
\item the \emph{optimal residual} $z^v$, whose norm $\|z^v\|_p = d(v, A;\, \|\cdot\|_p)$ is the $\ell_p$-distance from $v$ to $A$;
\item the \emph{support point} $y^v := v + z^v \in \operatorname{bd} A$;
\item the \emph{cut normal} $\tilde{w}^v := w^v - \lambda^v \bar{w}$, where $(w^v, \lambda^v)$ is an optimal dual solution. By Proposition~\ref{prop:lp-gradient} and the strong duality results of \cite[Section~3]{Ararat2024}, when $z^v \neq 0$ the cut normal satisfies $\|\tilde{w}^v\|_{\pstar} = 1$ and is given by
\begin{equation}\label{eq:cut-normal}
\tilde{w}^v_i = \frac{|z^v_i|^{p-1} \operatorname{sgn}(z^v_i)}{\|z^v\|_p^{p-1}}, \quad i = 1, \ldots, q.
\end{equation}
\end{itemize}

The key computational distinction from the Euclidean case is the objective function of~\eqref{eq:Pv}:
\begin{itemize}
\item For $p = 2$: $\|z\|_2$ is a second-order cone (SOC) objective, and~\eqref{eq:Pv} reduces to a second-order cone program (SOCP) or, after squaring, a quadratic program. Standard QP/SOCP solvers apply directly.
\item For $p \neq 2$: $\|z\|_p$ is convex but not SOC-representable in general. However, the $\ell_p$ norm can be represented using a system of second-order and power cone constraints, which modern conic solvers (MOSEK, SCS) handle natively. Disciplined convex programming frameworks such as CVX~\cite{cvx} generate these representations automatically from the expression $\|z\|_p$.
\end{itemize}

\subsection{The Algorithm}\label{sec:algorithm-description}

Algorithm~\ref{alg:lp-outer} presents the $\ell_p$ norm-minimization outer approximation procedure. The structure follows \cite[Algorithm~1]{Ararat2024}, with $\|\cdot\|_p$ replacing the Euclidean norm throughout.

\begin{algorithm}[H]
\caption{$\ell_p$ Norm-Minimization Outer Approximation}\label{alg:lp-outer}
\begin{algorithmic}[1]
\Input{CVOP $(\Gamma, \mathcal{X}, C)$, norm exponent $p \in (1,\infty)$, tolerance $\varepsilon > 0$, slice parameters $(\bar{w}, \gamma)$.}
\Initialize{Compute initial halfspaces by solving weighted-sum problems $\text{WS}(\omega^i)$ for $i = 1, \ldots, I$, where $\omega^1, \ldots, \omega^I \in C^+ \setminus \{0\}$ are chosen so that $P_0 := \bigcap_{i=1}^{I} H(\tilde{w}^i, A)$ satisfies $P_0 \supseteq A$ and $P_0$ is bounded.}
\For{$k = 0, 1, 2, \ldots$}
    \State Compute $V_k := \ext(P_k)$ \Comment{Vertex enumeration}
    \For{each $v \in V_k$}
        \State Solve~\eqref{eq:Pv} to obtain $(x^v, z^v)$ and set $y^v \gets v + z^v$
    \EndFor
    \State $v^k \gets \arg\max_{v \in V_k} \|z^v\|_p$ \Comment{Farthest vertex}
    \If{$\|z^{v^k}\|_p \le \varepsilon$}
        \State \textbf{return} $P_k$ \Comment{$\varepsilon$-solution found}
    \EndIf
    \State Compute cut normal $\tilde{w}^{v^k}$ via~\eqref{eq:cut-normal}
    \State $P_{k+1} \gets P_k \cap H(\tilde{w}^{v^k}, y^{v^k})$ \Comment{Cut}
\EndFor
\Output{Outer approximation $P_k$ with $\delta_H(P_k, A;\, \|\cdot\|_p) \le \varepsilon$.}
\end{algorithmic}
\end{algorithm}

By Proposition~\ref{prop:algorithm-h-sequence}, the sequence $(P_k)_{k \ge 0}$ is an $H(1,A)$-sequence of cutting. Theorem~\ref{thm:h-convergence} guarantees convergence, and Theorem~\ref{thm:lp-rate} establishes the rate $\delta_H(P_k, A) = O(k^{2/(1-q)})$ for all $p \in (1,\infty)$.

\begin{remark}[Validity of the algorithm for general $\ell_p$ norms]\label{rem:general-lp-validity}
The foundational results of Ararat et al.\ \cite{Ararat2024} that underpin Algorithm~\ref{alg:lp-outer}---strong duality of the scalarization \cite[Proposition~3.2]{Ararat2024}, the supporting halfspace property \cite[Proposition~3.6]{Ararat2024}, correctness \cite[Theorem~4.5]{Ararat2024}, finiteness \cite[Theorem~5.2]{Ararat2024}, the $H(1,A)$-sequence property \cite[Corollary~6.5]{Ararat2024}, and the general convergence rate $O(k^{1/(1-q)})$ \cite[Corollary~6.10]{Ararat2024}---are all established for an arbitrary norm. In particular, the proofs in \cite[Sections~3--6]{Ararat2024} rely only on convexity of $\|\cdot\|$, H\"older's inequality for the dual pair $(\|\cdot\|, \|\cdot\|_*)$, and volumetric packing of norm balls, none of which is specific to the Euclidean case. Only the improved rate $O(k^{2/(1-q)})$ in \cite[Section~7]{Ararat2024} requires Assumption~7.1 therein ($\|\cdot\| = \|\cdot\|_2$), which our Lemmas~\ref{lem:hyperplane-distance}--\ref{lem:packing} and Theorem~\ref{thm:lp-rate} replace for all $\ell_p$ norms with $p \in (1,\infty)$. The additional properties needed for the improved rate---strict convexity (ensuring uniqueness of the projection and cut normal) and differentiability (yielding the gradient formula \eqref{eq:cut-normal})---both hold for $\ell_p$ norms with $1 < p < \infty$.
\end{remark}

\section{Numerical Experiments}\label{sec:experiments}

We present numerical experiments that confirm the $p$-independent convergence rate predicted by Theorem~\ref{thm:lp-rate}. All experiments use Algorithm~\ref{alg:lp-outer} with the full vertex enumeration strategy.

\subsection{Test Problems}\label{sec:test-problems}

We test on three problems with known Pareto fronts, allowing precise evaluation of the Hausdorff approximation error.

The first test problem is Example~1 from \cite[Section~8]{Ararat2024}:
\begin{equation}\label{eq:example1}
\min\, \Gamma(x) = x \;\text{ w.r.t.\ } \le_{\R^q_+} \quad \text{s.t. } \|x - e\|_2 \le 1,
\end{equation}
where $e = (1, \ldots, 1)^\top \in \R^q$. The Pareto front is a portion of the sphere $\|y - e\|_2 = 1$ in the positive orthant. For $q = 2$, this is a quarter-circle centered at $(1,1)$; for $q = 3$, it is a spherical cap. The weighted-sum scalarization admits a closed-form solution $x^* = e - w/\|w\|_2$. We test with $q \in \{2, 3\}$.

The second test problem is a rotated ellipse:
\begin{equation}\label{eq:rotated-ellipse}
\min\, \Gamma(x) = x \;\text{ w.r.t.\ } \le_{\R^2_+} \quad \text{s.t. } \frac{(x_2 - x_1)^2}{10} + \frac{(x_1 + x_2 - 4)^2}{6} \le 1.
\end{equation}
The feasible set is an ellipse centered at $(2, 2)$ with semi-axes $\sqrt{10}$ and $\sqrt{6}$, rotated $45°$ from the coordinate axes. The non-isotropic curvature of the boundary provides a richer test case than Example~1, where the Pareto front has constant curvature. We test with $q = 2$.

The third test problem is Example~2 from \cite[Section~8]{Ararat2024}:
\begin{equation}\label{eq:example2}
\min\, \Gamma(x) = \bigl(\|x - a_1\|_2^2,\; \|x - a_2\|_2^2,\; \|x - a_3\|_2^2\bigr)^\top \;\text{ w.r.t.\ } \le_{\R^3_+}
\end{equation}
subject to $x_1 + 2x_2 \le 10$, $0 \le x_1 \le 10$, $0 \le x_2 \le 4$, where $a_1 = (1,1)^\top$, $a_2 = (2,3)^\top$, $a_3 = (4,2)^\top$. Unlike Example~1, the objective is nonlinear (sum of squared distances to three reference points) and the feasible set is a polytope rather than a ball. This provides a test case with $q = 3$ whose geometry is qualitatively different from Example~1.

\subsection{Setup}\label{sec:setup}

The algorithm is implemented in MATLAB. The $\ell_p$ subproblems~\eqref{eq:Pv} are solved using CVX~\cite{cvx}, which accepts $\|\cdot\|_p$ as a built-in atom and generates the appropriate conic representation automatically; for $p = 2$ this reduces to a standard SOCP, while for general $p$, CVX uses a combination of second-order and power cone constraints. Vertex enumeration ($\mathcal{H}$-to-$\mathcal{V}$ conversion at line~4 of Algorithm~\ref{alg:lp-outer}) is performed by the BENSOLVE Tools library~\cite{bensolve}. The cut normal~\eqref{eq:cut-normal} is computed directly from the optimal residual $z^v$ in $O(q)$ operations, and the dual norm identity $\|\tilde{w}^v\|_{\pstar} = 1$ (Proposition~\ref{prop:lp-gradient}) serves as a numerical check. Between iterations, subproblem solutions at non-cut vertices are cached and reused, reducing the per-iteration cost from $O(|V_k|)$ subproblem solves to $O(|\text{new vertices}|)$.

All experiments were run on a desktop PC with an Intel Core i9-9900K CPU (3.60\,GHz, 8 cores) and 32\,GB RAM, running Windows~11.

For each test problem, we run Algorithm~\ref{alg:lp-outer} with $p \in \{1.25, 1.5, 2, 3, 4, 8\}$, covering both $p < 2$ (where the direct $\ell_p$ approach would predict $c(p) = p < 2$) and $p > 2$ (where the dual exponent $\pstar < 2$ introduces further complications for naive approaches). The convergence tolerances are $\varepsilon = 10^{-4}$ for Example~1 with $q = 2$, $\varepsilon = 0.01$ for Example~1 with $q = 3$, $\varepsilon = 10^{-3}$ for the rotated ellipse, and $\varepsilon = 0.05$ for Example~2.

At each iteration $k$, we record the Hausdorff error $\delta_H(P_k, A;\, \|\cdot\|_p)$. To extract the empirical convergence rate exponent $\hat{c}(p)$, we fit the model
\begin{equation}\label{eq:rate-model}
\delta_H(P_k, A) \approx \Lambda \cdot k^{\hat{c}/(1-q)}
\end{equation}
by linear regression in log-log coordinates (i.e., $\log \delta_H$ vs.\ $\log k$) on the active decay region, excluding the initial transient and the terminal plateau near $\varepsilon$. We apply a monotone envelope ($\delta_H(P_k, A) \gets \min_{j \le k} \delta_H(P_j, A)$) to smooth non-monotonicity caused by vertex enumeration artifacts.

\subsection{Results}\label{sec:results}

For Example~1 with $q = 3$, Table~\ref{tab:example1-q3} and Figure~\ref{fig:convergence-q3} report the fitted rate exponents. The algorithm runs for $50$--$89$ iterations depending on $p$, providing sufficient data for rate estimation ($R^2 > 0.90$ for all fits).

\begin{table}[H]
\centering
\caption{Empirical convergence rate exponents for Example~1 with $q = 3$: $p$ is the norm index, $\pstar$ the dual exponent, $\hat{c}(p)$ the fitted rate exponent from~\eqref{eq:rate-model}, $R^2$ the coefficient of determination, and Iterations the total iteration count. The theoretical prediction is $c(p) = 2$ for all $p$.}\label{tab:example1-q3}
\begin{tabular}{@{}ccccc@{}}
\toprule
$p$ & $\pstar$ & $\hat{c}(p)$ & $R^2$ & Iterations \\
\midrule
$1.25$ & $5.00$ & $2.46$ & $0.920$ & $89$ \\
$1.50$ & $3.00$ & $2.38$ & $0.902$ & $82$ \\
$2.00$ & $2.00$ & $2.40$ & $0.915$ & $74$ \\
$3.00$ & $1.50$ & $2.40$ & $0.916$ & $59$ \\
$4.00$ & $1.33$ & $2.45$ & $0.914$ & $57$ \\
$8.00$ & $1.14$ & $2.51$ & $0.910$ & $50$ \\
\bottomrule
\end{tabular}
\end{table}

The empirical exponents cluster around $\hat{c} \approx 2.4$ with a spread of only $\hat{c}_{\max} - \hat{c}_{\min} = 0.13$ across all six values of $p$, consistent with the theoretical prediction $c(p) = 2$. The slight overestimate ($\hat{c} > 2$) is attributable to an initial plateau of $3$--$5$ iterations where the error barely decreases (visible in Figure~\ref{fig:convergence-q3}), which inflates the fitted slope.

\begin{figure}[H]
\centering
\includegraphics[width=0.8\textwidth]{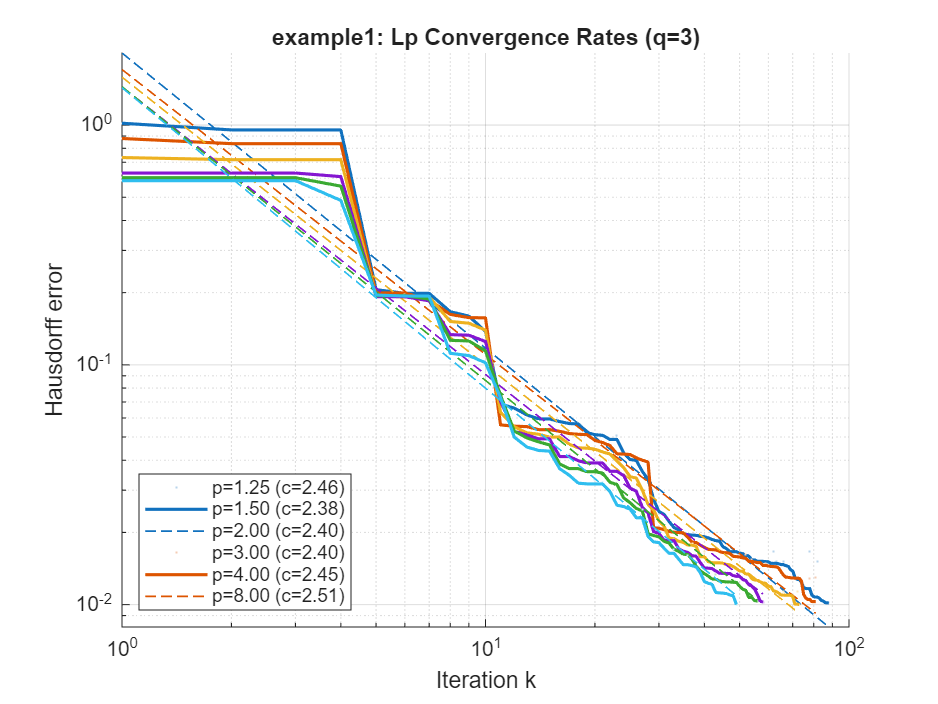}
\caption{Convergence of the Hausdorff error $\delta_H(P_k, A;\, \|\cdot\|_p)$ for Example~1 with $q=3$ and $p \in \{1.25, 1.5, 2, 3, 4, 8\}$. Solid lines show the monotone envelope; dashed lines show fitted power-law models in log-log coordinates. All curves decay at approximately the same rate, confirming $c(p) = 2$.}\label{fig:convergence-q3}
\end{figure}

For Example~1 with $q = 2$, Table~\ref{tab:example1-q2} and Figure~\ref{fig:convergence-q2} report results with a tighter tolerance $\varepsilon = 10^{-4}$ to ensure sufficient iterations for reliable rate estimation. The algorithm runs for $38$--$53$ iterations with $R^2 > 0.93$ for all fits.

\begin{table}[H]
\centering
\caption{Empirical convergence rate exponents for Example~1 with $q = 2$. The theoretical prediction is $c(p) = 2$ for all $p$.}\label{tab:example1-q2}
\begin{tabular}{@{}ccccc@{}}
\toprule
$p$ & $\pstar$ & $\hat{c}(p)$ & $R^2$ & Iterations \\
\midrule
$1.25$ & $5.00$ & $1.55$ & $0.935$ & $38$ \\
$1.50$ & $3.00$ & $1.56$ & $0.963$ & $45$ \\
$2.00$ & $2.00$ & $1.63$ & $0.959$ & $53$ \\
$3.00$ & $1.50$ & $1.63$ & $0.951$ & $49$ \\
$4.00$ & $1.33$ & $1.48$ & $0.947$ & $43$ \\
$8.00$ & $1.14$ & $1.47$ & $0.966$ & $51$ \\
\bottomrule
\end{tabular}
\end{table}

The empirical exponents cluster around $\hat{c} \approx 1.55$ (spread $0.16$), showing the same $p$-independence as the $q = 3$ case. The systematic underestimate relative to $c = 2$ is a finite-sample effect common to all $q = 2$ experiments: since each cut reduces the error by a larger fraction when $q = 2$ (the theoretical rate is $O(k^{-2})$ vs.\ $O(k^{-1})$ for $q = 3$), the algorithm terminates in fewer iterations and the log-log fit captures the transient rather than the asymptotic slope. Even the Euclidean case $p = 2$ yields only $\hat{c}(2) = 1.63$, confirming that the gap is not $p$-specific.

\begin{figure}[H]
\centering
\includegraphics[width=0.8\textwidth]{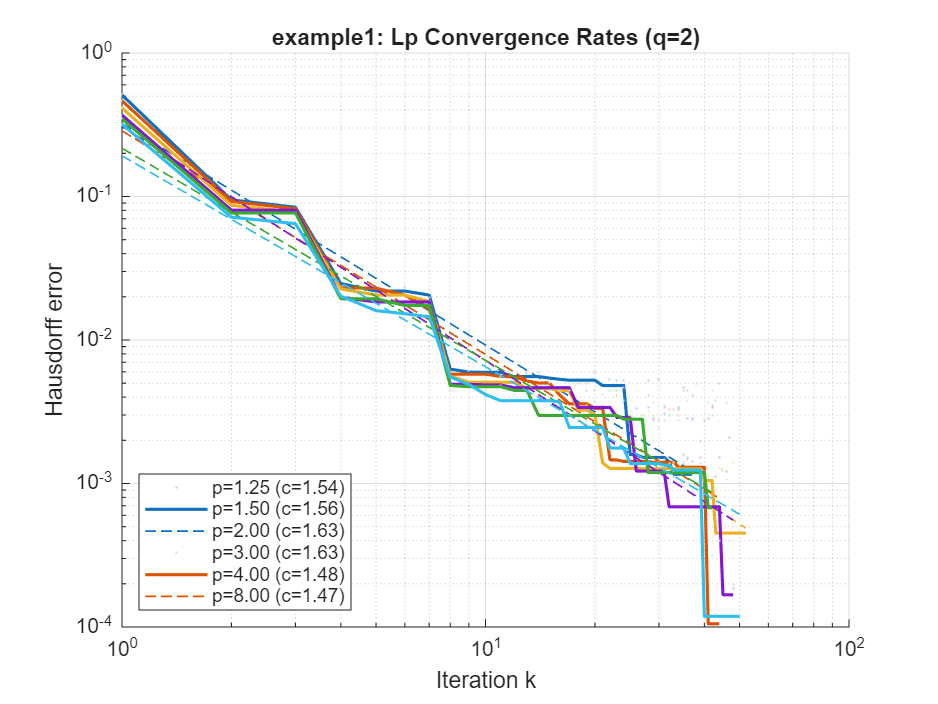}
\caption{Convergence of the Hausdorff error $\delta_H(P_k, A;\, \|\cdot\|_p)$ for Example~1 with $q=2$ and $p \in \{1.25, 1.5, 2, 3, 4, 8\}$. The convergence tolerance is $\varepsilon = 10^{-4}$.}\label{fig:convergence-q2}
\end{figure}

For the rotated ellipse with $q = 2$, Table~\ref{tab:rotated-ellipse} reports results with $\varepsilon = 10^{-3}$. The tighter tolerance yields $35$--$45$ iterations per run with $R^2 > 0.93$ for all fits.

\begin{table}[H]
\centering
\caption{Empirical convergence rate exponents for the rotated ellipse with $q = 2$. The theoretical prediction is $c(p) = 2$ for all $p$.}\label{tab:rotated-ellipse}
\begin{tabular}{@{}ccccc@{}}
\toprule
$p$ & $\pstar$ & $\hat{c}(p)$ & $R^2$ & Iterations \\
\midrule
$1.25$ & $5.00$ & $1.54$ & $0.949$ & $42$ \\
$1.50$ & $3.00$ & $1.61$ & $0.939$ & $40$ \\
$2.00$ & $2.00$ & $1.58$ & $0.970$ & $45$ \\
$3.00$ & $1.50$ & $1.75$ & $0.975$ & $44$ \\
$4.00$ & $1.33$ & $1.62$ & $0.971$ & $35$ \\
$8.00$ & $1.14$ & $1.43$ & $0.952$ & $44$ \\
\bottomrule
\end{tabular}
\end{table}

The empirical exponents range from $\hat{c} = 1.43$ to $1.75$ (spread $0.32$), again $p$-independent but systematically below $c = 2$. The pattern mirrors Example~1 with $q = 2$: moderate iteration counts ($35$--$45$) mean the log-log fit reflects the transient regime. The Euclidean case $p = 2$ yields $\hat{c}(2) = 1.58$, confirming that the underestimate is not $p$-specific.

For Example~2 with $q = 3$, Table~\ref{tab:example2-q3} and Figure~\ref{fig:convergence-example2} report results with $\varepsilon = 0.05$. Unlike Example~1, the objective is nonlinear and the feasible set is a polytope, providing an independent test of the $p$-independence of the rate exponent. The algorithm runs for $41$--$72$ iterations with $R^2 > 0.92$ for all fits.

\begin{table}[H]
\centering
\caption{Empirical convergence rate exponents for Example~2 with $q = 3$. The theoretical prediction is $c(p) = 2$ for all $p$.}\label{tab:example2-q3}
\begin{tabular}{@{}ccccc@{}}
\toprule
$p$ & $\pstar$ & $\hat{c}(p)$ & $R^2$ & Iterations \\
\midrule
$1.25$ & $5.00$ & $2.74$ & $0.969$ & $72$ \\
$1.50$ & $3.00$ & $2.69$ & $0.966$ & $61$ \\
$2.00$ & $2.00$ & $2.66$ & $0.964$ & $56$ \\
$3.00$ & $1.50$ & $2.75$ & $0.940$ & $49$ \\
$4.00$ & $1.33$ & $2.72$ & $0.930$ & $45$ \\
$8.00$ & $1.14$ & $2.72$ & $0.928$ & $41$ \\
\bottomrule
\end{tabular}
\end{table}

The empirical exponents cluster tightly around $\hat{c} \approx 2.72$ with the smallest spread of all experiments ($0.09$). As with Example~1 at $q = 3$, the fitted exponents overestimate $c = 2$, attributable to the initial plateau visible in Figure~\ref{fig:convergence-example2}. The tighter clustering here, on a problem with nonlinear objectives and a polytope feasible set, provides further evidence that the rate exponent does not depend on $p$.

\begin{figure}[H]
\centering
\includegraphics[width=0.8\textwidth]{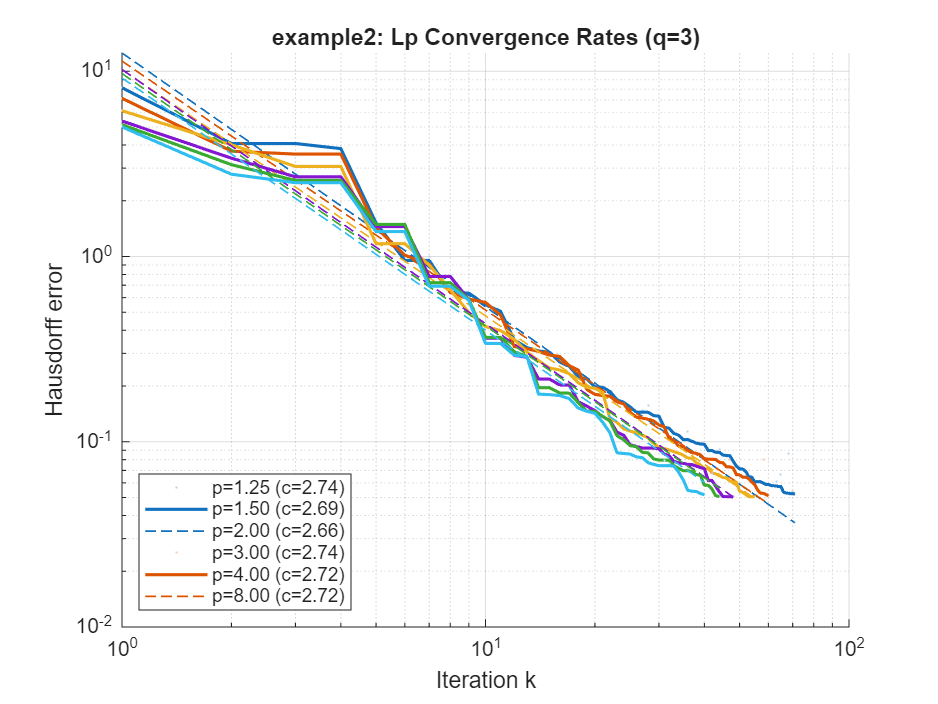}
\caption{Convergence of the Hausdorff error $\delta_H(P_k, A;\, \|\cdot\|_p)$ for Example~2 with $q=3$, $p \in \{1.25, 1.5, 2, 3, 4, 8\}$, and $\varepsilon = 0.05$.}\label{fig:convergence-example2}
\end{figure}

\subsection{Discussion}\label{sec:discussion}

The experiments support the main theoretical result.

Across all four experimental configurations (Example~1 with $q = 2$ and $q = 3$, the rotated ellipse with $q = 2$, and Example~2 with $q = 3$) and all six values of $p$, the fitted exponent $\hat{c}(p)$ shows no systematic dependence on $p$. Within each configuration, the standard deviation of $\hat{c}$ across $p$ is at most $0.1$, and the spread $\hat{c}_{\max} - \hat{c}_{\min}$ ranges from $0.09$ (Example~2) to $0.32$ (rotated ellipse). The earlier conjecture $c(p) = \min(p, p/(p-1))$, which predicts $c(2) = 2$ and $c(p) \to 1$ as $p \to 1^+$ or $p \to \infty$, is contradicted by the flat profile of $\hat{c}$ across $p$.

The absolute value of $\hat{c}$ depends on the number of objectives $q$. For $q = 3$, the fitted exponents overestimate the theoretical $c = 2$: $\hat{c} \approx 2.4$ for Example~1 and $\hat{c} \approx 2.7$ for Example~2. This overestimate is caused by an initial plateau of $3$--$5$ iterations where the error barely decreases, which steepens the fitted slope. For $q = 2$, the fitted exponents underestimate $c = 2$: $\hat{c} \approx 1.55$ for Example~1 and $\hat{c} \approx 1.6$ for the rotated ellipse. This underestimate arises because the theoretical rate $O(k^{-2})$ for $q = 2$ is faster, so the algorithm terminates in fewer iterations and the log-log fit captures transient behavior rather than the asymptotic slope. In both cases, the Euclidean run ($p = 2$) shows the same bias, confirming that the deviation from $c = 2$ is a finite-sample effect, not a $p$-dependent one.

The convergence \emph{constants} do depend on $p$: in Example~1 with $q = 3$, the algorithm converges in $89$ iterations for $p = 1.25$ but only $50$ for $p = 8$. This is consistent with the theoretical constant $\Lambda(p,q,R,A)$ in Theorem~\ref{thm:lp-rate}, which depends on $p$ through $N_{2,p}$ and the packing constant. The Euclidean case $p = 2$ achieves $N_{2,2} = 1$ and the smallest constants.

\section{Conclusion}\label{sec:conclusion}

We have established that the convergence rate exponent of the norm-minimization-based outer approximation algorithm for convex vector optimization is $c(p) = 2$ for every $\ell_p$ norm with $p \in (1,\infty)$, matching the optimal Euclidean rate $O(k^{2/(1-q)})$. This result resolves the question of how the choice of scalarization norm affects convergence speed within the $\ell_p$ family: the exponent is entirely $p$-independent.

The proof relies on the \emph{Euclidean intermediary technique} introduced in Lemma~\ref{lem:hyperplane-distance}. A direct analysis via the modulus of smoothness of $\|\cdot\|_p$ yields only $c(p) = \min(p,2)$, which degrades for $1 < p < 2$. Instead, we expand the squared Euclidean norm of the deviation vector difference and exploit the non-negativity of all resulting terms under the support conditions. Converting to the $\ell_p$ metric via norm equivalence costs only a dimension-dependent constant, preserving the quadratic exponent.

Together with \cite{Alshahrani2026}, which proves $c = 2$ for all inner-product norms $\|\cdot\|_M$, the present work shows that the rate exponent~$2$ is robust across two broad families of norms: inner-product norms (where the proof uses isometry invariance) and $\ell_p$ norms (where the proof uses the Euclidean intermediary). In both cases, the ambient Euclidean structure of $\R^q$ plays a decisive role, suggesting that the exponent~$2$ may be a universal feature of the algorithm rather than an artifact of any particular norm geometry.

Several directions remain open:
\begin{enumerate}[label=(\roman*)]
\item Does $c = 2$ hold for \emph{every} strictly convex, smooth norm on $\R^q$? The Euclidean intermediary argument does not require any special property of the scalarization norm beyond strict convexity and differentiability; the $\ell_p$ structure is used only in the packing estimates and the norm equivalence constant. An extension to arbitrary norms would settle this question.

\item The $\ell_1$ and $\ell_\infty$ norms lack strict convexity and smoothness, respectively, so the algorithm's cut uniqueness (Lemma~\ref{lem:strict-convexity}) fails. Whether a modified algorithm can achieve algebraic convergence rates at these endpoints is unclear.

\item While the exponent is $p$-independent, the constant $\Lambda(p,q,R,A)$ grows with the norm equivalence factor $N_{2,p}$ (Remark~\ref{rem:norm-equiv-role}). Our numerical experiments confirm that the Euclidean case $p = 2$ achieves the best constants. Determining the sharp dependence of $\Lambda$ on $p$ and $q$, and whether the norm equivalence bound is tight, remains open.

\item Since $p$ affects only the constants, one could in principle select or adapt the scalarization norm during the algorithm to minimize the approximation constant. How adaptive norm strategies interact with the H-sequence framework is an open question.
\end{enumerate}

\section*{Funding}
This research received no external funding.

\section*{Data and Code Availability}
No external data were used in this study. All numerical results were generated by the algorithms described in the paper. The code used to produce the numerical experiments is available from the corresponding author upon reasonable request.

\section*{Acknowledgements}

The author acknowledges the support of King Fahd University of Petroleum \& Minerals (KFUPM) and the Interdisciplinary Research Center for Smart Mobility and Logistics at KFUPM.

\printbibliography

@misc{Alshahrani2026,
	title = {Adaptive {Metrics} for {Norm}-{Minimization}-{Based} {Outer} {Approximation} in {Convex} {Vector} {Optimization}},
	url = {http://arxiv.org/abs/2605.14320},
	doi = {10.48550/arXiv.2605.14320},
	urldate = {2026-05-15},
	publisher = {arXiv},
	author = {Alshahrani, Mohammed},
	month = may,
	year = {2026},
}

@article{Glasauer1997,
	title = {Asymptotic estimates for best and stepwise approximation of convex bodies {III}},
	volume = {9},
	issn = {0933-7741, 1435-5337},
	url = {https://www.degruyter.com/document/doi/10.1515/form.1997.9.383/html},
	doi = {10.1515/form.1997.9.383},
	language = {en},
	number = {9},
	urldate = {2026-02-12},
	journal = {Forum Mathematicum},
	author = {Glasauer, Stefan and Gruber, Peter M.},
	year = {1997},
	pages = {383--404},
}

@book{Lohne2011,
	address = {Berlin, Heidelberg},
	series = {Vector {Optimization}},
	title = {Vector {Optimization} with {Infimum} and {Supremum}},
	copyright = {https://www.springernature.com/gp/researchers/text-and-data-mining},
	url = {https://link.springer.com/10.1007/978-3-642-18351-5},
	doi = {10.1007/978-3-642-18351-5},
	language = {en},
	urldate = {2026-01-24},
	publisher = {Springer Berlin Heidelberg},
	author = {Löhne, Andreas},
	year = {2011},
}

@article{Gruber1993,
	title = {Asymptotic estimates for best and stepwise approximation of convex bodies {II}},
	volume = {5},
	url = {https://doi.org/10.1515/form.1993.5.521},
	doi = {10.1515/form.1993.5.521},
	number = {Jahresband},
	urldate = {2026-02-12},
	journal = {Forum Mathematicum},
	author = {Gruber, Peter M},
	year = {1993},
	pages = {521--538},
}

@article{Clarkson1936,
	title = {Uniformly convex spaces},
	volume = {40},
	number = {3},
	journal = {Transactions of the American Mathematical Society},
	author = {Clarkson, James A},
	year = {1936},
	pages = {396--414},
}

@article{Wagner2023,
	title = {Algorithms to {Solve} {Unbounded} {Convex} {Vector} {Optimization} {Problems}},
	volume = {33},
	doi = {10.1137/22M1507693},
	number = {4},
	journal = {SIAM Journal on Optimization},
	author = {Wagner, Alexander and Ulus, Firdevs and Rudloff, Birgit and Kováčová, Gabriela and Hey, Norman},
	year = {2023},
	pages = {2598--2624},
}

@article{Sprung2019,
	title = {Upper and lower bounds for the {Bregman} divergence},
	volume = {2019},
	issn = {1029-242X},
	url = {https://doi.org/10.1186/s13660-018-1953-y},
	doi = {10.1186/s13660-018-1953-y},
	language = {en},
	number = {1},
	urldate = {2026-02-23},
	journal = {Journal of Inequalities and Applications},
	author = {Sprung, Benjamin},
	month = jan,
	year = {2019},
	pages = {4},
}

@article{bensolve,
	title = {The {Vector} {Linear} {Program} {Solver} {Bensolve} – {Notes} on {Theoretical} {Background}},
	volume = {260},
	doi = {10.1016/j.ejor.2016.02.039},
	number = {3},
	journal = {European Journal of Operational Research},
	author = {Löhne, Andreas and Weißing, Benjamin},
	year = {2017},
	pages = {807--813},
}

@misc{cvx,
	title = {{CVX}: {MATLAB} {Software} for {Disciplined} {Convex} {Programming}, {Version} 2.2},
	url = {http://cvxr.com/cvx},
	author = {Grant, Michael and Boyd, Stephen},
	month = mar,
	year = {2020},
}

@book{Lotov2004,
	address = {Boston, MA},
	series = {Applied {Optimization}},
	title = {Interactive {Decision} {Maps}},
	volume = {89},
	copyright = {http://www.springer.com/tdm},
	isbn = {978-1-4613-4690-6 978-1-4419-8851-5},
	url = {http://link.springer.com/10.1007/978-1-4419-8851-5},
	doi = {10.1007/978-1-4419-8851-5},
	urldate = {2026-02-23},
	publisher = {Springer US},
	author = {Lotov, Alexander V. and Bushenkov, Vladimir A. and Kamenev, Georgy K.},
	editor = {Pardalos, Panos M. and Hearn, Donald W.},
	year = {2004},
}

@article{Bronstein2008,
	title = {Approximation of convex sets by polytopes},
	volume = {153},
	issn = {1573-8795},
	url = {https://doi.org/10.1007/s10958-008-9144-x},
	doi = {10.1007/s10958-008-9144-x},
	language = {en},
	number = {6},
	urldate = {2026-02-23},
	journal = {Journal of Mathematical Sciences},
	author = {Bronstein, E. M.},
	month = sep,
	year = {2008},
	pages = {727--762},
}

@book{Schneider2013,
	address = {Cambridge},
	edition = {2},
	series = {Encyclopedia of {Mathematics} and its {Applications}},
	title = {Convex {Bodies}: {The} {Brunn}–{Minkowski} {Theory}},
	isbn = {978-1-107-60101-7},
	shorttitle = {Convex {Bodies}},
	url = {https://www.cambridge.org/core/books/convex-bodies-the-brunnminkowski-theory/400F6173EE613859F144E9598DDD8BDF},
	doi = {10.1017/CBO9781139003858},
	urldate = {2026-02-23},
	publisher = {Cambridge University Press},
	author = {Schneider, Rolf},
	year = {2013},
}

@article{Gruber1993a,
	chapter = {Forum Mathematicum},
	title = {Asymptotic estimates for best and stepwise approximation of convex bodies {I}},
	volume = {5},
	copyright = {De Gruyter expressly reserves the right to use all content for commercial text and data mining within the meaning of Section 44b of the German Copyright Act.},
	issn = {1435-5337},
	url = {https://www.degruyterbrill.com/document/doi/10.1515/form.1993.5.281/html},
	doi = {10.1515/form.1993.5.281},
	language = {en},
	number = {Jahresband},
	urldate = {2026-02-12},
	publisher = {De Gruyter},
	author = {Gruber, Peter M.},
	month = jan,
	year = {1993},
	pages = {281--298},
}

@article{Boroczky2000,
	title = {The error of polytopal approximation with respect to the symmetric difference metric and {theLpmetric}},
	volume = {117},
	issn = {1565-8511},
	url = {https://doi.org/10.1007/BF02773561},
	doi = {10.1007/BF02773561},
	language = {en},
	number = {1},
	urldate = {2026-02-23},
	journal = {Israel Journal of Mathematics},
	author = {Böröczky, Károly},
	month = dec,
	year = {2000},
	pages = {1--28},
}

@article{Ball1994,
	title = {Sharp uniform convexity and smoothness inequalities for trace norms},
	volume = {115},
	issn = {1432-1297},
	url = {https://doi.org/10.1007/BF01231769},
	doi = {10.1007/BF01231769},
	language = {en},
	number = {1},
	urldate = {2026-02-23},
	journal = {Inventiones mathematicae},
	author = {Ball, Keith and Carlen, Eric A. and Lieb, Elliott H.},
	month = dec,
	year = {1994},
	pages = {463--482},
}

@article{Kamenev2002,
	title = {Conjugate {Adaptive} {Algorithms} for {Polyhedral} {Approximation} of {Convex} {Bodies}},
	volume = {42},
	number = {9},
	journal = {Computational Mathematics and Mathematical Physics},
	author = {Kamenev, G. K.},
	year = {2002},
	pages = {1301--1316},
}

@article{Kamenev1992,
	title = {A {Class} of {Adaptive} {Algorithms} for {Approximating} {Convex} {Bodies} by {Polyhedra}},
	volume = {32},
	number = {1},
	journal = {Computational Mathematics and Mathematical Physics},
	author = {Kamenev, G. K.},
	year = {1992},
	pages = {114--127},
}

@article{Hanner1956,
	title = {On the uniform convexity {ofLpandlp}},
	volume = {3},
	issn = {1871-2487},
	url = {https://doi.org/10.1007/BF02589410},
	doi = {10.1007/BF02589410},
	language = {en},
	number = {3},
	urldate = {2026-02-12},
	journal = {Arkiv för Matematik},
	author = {Hanner, Olof},
	month = feb,
	year = {1956},
	pages = {239--244},
}

@article{Xu1991,
	title = {Characteristic inequalities of uniformly convex and uniformly smooth {Banach} spaces},
	volume = {157},
	issn = {0022-247X},
	url = {https://www.sciencedirect.com/science/article/pii/0022247X9190144O},
	doi = {10.1016/0022-247X(91)90144-O},
	number = {1},
	urldate = {2026-02-23},
	journal = {Journal of Mathematical Analysis and Applications},
	author = {Xu, Zong-Ben and Roach, G. F},
	month = may,
	year = {1991},
	pages = {189--210},
}

@incollection{Gruber1994,
	address = {Dordrecht},
	title = {Approximation by {Convex} {Polytopes}},
	isbn = {978-94-010-4398-4 978-94-011-0924-6},
	url = {http://link.springer.com/10.1007/978-94-011-0924-6_8},
	doi = {10.1007/978-94-011-0924-6_8},
	urldate = {2026-02-12},
	booktitle = {Polytopes: {Abstract}, {Convex} and {Computational}},
	publisher = {Springer Netherlands},
	author = {Gruber, P. M.},
	editor = {Bisztriczky, T. and McMullen, P. and Schneider, R. and Weiss, A. Ivić},
	year = {1994},
	doi = {10.1007/978-94-011-0924-6_8},
	pages = {173--203},
}

@article{Ararat2022,
	title = {A {Norm} {Minimization}-{Based} {Convex} {Vector} {Optimization} {Algorithm}},
	volume = {194},
	issn = {1573-2878},
	url = {https://doi.org/10.1007/s10957-022-02045-8},
	doi = {10.1007/s10957-022-02045-8},
	language = {en},
	number = {2},
	urldate = {2026-01-18},
	journal = {Journal of Optimization Theory and Applications},
	author = {Ararat, Çağın and Ulus, Firdevs and Umer, Muhammad},
	month = aug,
	year = {2022},
	pages = {681--712},
}

@book{Lindenstrauss1996,
	address = {Berlin, Heidelberg},
	series = {Classics in {Mathematics}},
	title = {Classical {Banach} {Spaces} {I} and {II}: {Sequence} {Spaces} and {Function} {Spaces}},
	copyright = {https://www.springer.com/tdm},
	isbn = {978-3-540-60628-4 978-3-662-53294-2},
	shorttitle = {Classical {Banach} {Spaces} {I} and {II}},
	url = {https://link.springer.com/10.1007/978-3-662-53294-2},
	doi = {10.1007/978-3-662-53294-2},
	language = {en},
	urldate = {2026-02-12},
	publisher = {Springer},
	author = {Lindenstrauss, Joram and Tzafriri, Lior},
	year = {1996},
}

@article{Keskin2023,
	title = {Outer {Approximation} {Algorithms} for {Convex} {Vector} {Optimization} {Problems}},
	volume = {38},
	issn = {1055-6788},
	url = {https://doi.org/10.1080/10556788.2023.2167994},
	doi = {10.1080/10556788.2023.2167994},
	number = {4},
	urldate = {2026-01-19},
	journal = {Optimization Methods and Software},
	publisher = {Taylor \& Francis},
	author = {Keskin, İrem Nur and Ulus, Firdevs},
	month = jul,
	year = {2023},
	pages = {723--755},
}

@article{Eichfelder2026,
	title = {Local {Upper} {Bounds} {Based} on {Polyhedral} {Ordering} {Cones}},
	volume = {14},
	issn = {2192-4406},
	url = {https://www.sciencedirect.com/science/article/pii/S2192440625000218},
	doi = {10.1016/j.ejco.2025.100124},
	urldate = {2026-01-18},
	journal = {EURO Journal on Computational Optimization},
	author = {Eichfelder, Gabriele and Ulus, Firdevs},
	month = jan,
	year = {2026},
	pages = {100124},
}

@article{Eichfelder2022,
	title = {An {Approximation} {Algorithm} for {Multi}-{Objective} {Optimization} {Problems} {Using} a {Box}-{Coverage}},
	volume = {83},
	issn = {0925-5001, 1573-2916},
	url = {https://link.springer.com/10.1007/s10898-021-01109-9},
	doi = {10.1007/s10898-021-01109-9},
	language = {en},
	number = {2},
	urldate = {2026-01-24},
	journal = {Journal of Global Optimization},
	author = {Eichfelder, Gabriele and Warnow, Leo},
	month = jun,
	year = {2022},
	pages = {329--357},
}

@article{Lohne2014,
	title = {Primal and {Dual} {Approximation} {Algorithms} for {Convex} {Vector} {Optimization} {Problems}},
	volume = {60},
	copyright = {http://www.springer.com/tdm},
	issn = {0925-5001, 1573-2916},
	url = {http://link.springer.com/10.1007/s10898-013-0136-0},
	doi = {10.1007/s10898-013-0136-0},
	language = {en},
	number = {4},
	urldate = {2024-04-01},
	journal = {Journal of Global Optimization},
	author = {Löhne, Andreas and Rudloff, Birgit and Ulus, Firdevs},
	month = dec,
	year = {2014},
	pages = {713--736},
}

@article{Ehrgott2011,
	title = {An {Approximation} {Algorithm} for {Convex} {Multi}-{Objective} {Programming} {Problems}},
	volume = {50},
	issn = {0925-5001, 1573-2916},
	url = {https://link.springer.com/10.1007/s10898-010-9588-7},
	doi = {10.1007/s10898-010-9588-7},
	language = {en},
	number = {3},
	urldate = {2026-01-24},
	journal = {Journal of Global Optimization},
	author = {Ehrgott, Matthias and Shao, Lizhen and Schöbel, Anita},
	month = jul,
	year = {2011},
	pages = {397--416},
}

@article{Gruber1982,
	title = {Approximation of {Convex} {Bodies} by {Polytopes}},
	volume = {31},
	issn = {1973-4409},
	url = {https://doi.org/10.1007/BF02844354},
	doi = {10.1007/BF02844354},
	language = {en},
	number = {2},
	urldate = {2026-01-26},
	journal = {Rendiconti del Circolo Matematico di Palermo},
	author = {Gruber, Peter M. and Kenderov, Petar},
	month = jun,
	year = {1982},
	pages = {195--225},
}

@article{Benson1998,
	title = {An {Outer} {Approximation} {Algorithm} for {Generating} {All} {Efficient} {Extreme} {Points} in the {Outcome} {Set} of a {Multiple} {Objective} {Linear} {Programming} {Problem}},
	volume = {13},
	issn = {1573-2916},
	url = {https://doi.org/10.1023/A:1008215702611},
	doi = {10.1023/A:1008215702611},
	number = {1},
	journal = {Journal of Global Optimization},
	author = {Benson, Harold P.},
	month = jan,
	year = {1998},
	pages = {1--24},
}

@article{Ararat2024,
	title = {Convergence {Analysis} of a {Norm} {Minimization}-{Based} {Convex} {Vector} {Optimization} {Algorithm}},
	volume = {34},
	issn = {1052-6234},
	url = {https://epubs.siam.org/doi/abs/10.1137/23M1574580},
	doi = {10.1137/23M1574580},
	number = {3},
	urldate = {2026-01-18},
	journal = {SIAM Journal on Optimization},
	publisher = {Society for Industrial and Applied Mathematics},
	author = {Ararat, Çağin and Ulus, Firdevs and Umer, Muhammad},
	month = sep,
	year = {2024},
	pages = {2700--2728},
}

\end{document}